\documentclass[10pt]{amsart}

\usepackage{amssymb,amsmath,amsthm}
\usepackage[all]{xy}
\usepackage{eucal}

\usepackage{color}

\theoremstyle{plain}
\newtheorem{thm}[subsection]{Theorem}

\newtheorem{prop}[subsection]{Proposition}

\theoremstyle{definition}
\newtheorem{defn}[subsection]{Definition}

\theoremstyle{remark}
\newtheorem{rem}[subsection]{Remark}

\makeatletter
\let\c@equation\c@subsection

\makeatother

\newcommand{\NN}{{ \mathbb{N} }}

\newcommand{\capC}{{ \mathcal{C} }}

\newcommand{\capP}{{ \mathcal{P} }}

\newcommand{\capX}{{ \mathcal{X} }}
\newcommand{\capY}{{ \mathcal{Y} }}

\newcommand{\Ev}{{ \mathrm{Ev} }}

\newcommand{\id}{{ \mathrm{id} }}

\newcommand{\Space}{{ \mathsf{S} }}

\newcommand{\Ho}{{ \mathsf{Ho} }}

\newcommand{\sSet}{{ \mathsf{sSet} }}

\newcommand{\Mod}{{ \mathsf{Mod} }}

\newcommand{\M}{{ \mathsf{M} }}

\newcommand{\Set}{{ \mathsf{Set} }}

\newcommand{\K}{{ \mathsf{K} }}
\newcommand{\Kt}{{ \tilde{\mathsf{K}} }}

\newcommand{\coAlg}{{ \mathsf{coAlg} }}
\newcommand{\res}{{ \mathsf{res} }}
\newcommand{\BK}{{ \mathsf{BK} }}
\newcommand{\CGHaus}{{ \mathsf{CGHaus} }}

\newcommand{\coAlgK}{{ \coAlg_\K }}
\newcommand{\coAlgKt}{{ \coAlg_{\tilde{\K}} }}

\newcommand{\Loop}{{ \Omega }}
\newcommand{\Loopt}{{ \tilde{\Omega} }}
\newcommand{\Susp}{{ \Sigma }}

\newcommand{\Smash}{{ \,\wedge\, }}
\newcommand{\tensor}{{ \otimes }}

\newcommand{\tensordot}{{ \dot{\tensor} }}

\newcommand{\wequiv}{{ \ \simeq \ }}

\newcommand{\Iso}{{  \ \cong \ }}

\newcommand{\rarrow}{{ \rightarrow }}

\newcommand{\function}[3]{{ {#1}\colon\thinspace{#2}\rarrow{#3} }}
\newcommand{\functionlong}[3]{{ {#1}\colon\thinspace{#2}\longrightarrow{#3} }}

\DeclareMathOperator*{\colim}{colim}
\DeclareMathOperator*{\holim}{holim}
\DeclareMathOperator{\Hombold}{\mathbf{Hom}}
\DeclareMathOperator{\hombold}{\mathbf{hom}}

\DeclareMathOperator{\Map}{Map}

\DeclareMathOperator{\Cobar}{Cobar}

\DeclareMathOperator{\Tot}{Tot}

\DeclareMathOperator{\hofib}{hofib}
\DeclareMathOperator{\iter}{iterated}

\title[Suspension spectra]{Suspension spectra and higher stabilization}

\author{Jacobson R. Blomquist}
\author{John E. Harper}

\address{Department of Mathematics, The Ohio State University, 231 West 18th Ave, Columbus, OH 43210, USA}
\email{blomquist.9@osu.edu}

\address{Department of Mathematics, The Ohio State University, Newark, 1179 University Dr, Newark, OH 43055, USA}
\email{harper.903@math.osu.edu}

\begin{document}

\begin{abstract}
We prove that the stabilization of spaces functor---the classical construction of associating a spectrum to a pointed space by tensoring with the sphere spectrum---satisfies homotopical descent on objects and morphisms. This is the stabilization analog of the Quillen-Sullivan theory main result that the rational chains (resp. cochains) functor participates in a derived equivalence with certain coalgebra (resp. algebra) complexes, after restriction to 1-connected spaces up to rational equivalence. In more detail, we prove that the stabilization of spaces functor participates in a derived equivalence with certain coalgebra spectra (where the stabilization construction naturally lands), after restriction to 1-connected spaces up to weak equivalence. This resolves in the affirmative the infinite case, involving stabilization and suspension spectra, of a question/conjecture posed by Tyler Lawson on (iterated) suspension spaces almost ten years ago. A key ingredient of our proof, of independent interest, is a higher stabilization theorem for spaces that provides strong estimates for the uniform cartesian-ness of certain cubical diagrams associated to $n$-fold iterations of the spaces-level stabilization map; this is the stabilization analog of Dundas' higher Hurewicz theorem.
\end{abstract}

\maketitle

\section{Introduction}

We have written this paper simplicially: in other words, ``space'' means ``simplicial set'' unless otherwise noted; see Dwyer-Henn \cite {Dwyer_Henn} for a useful introduction to these ideas, together with Bousfield-Kan \cite{Bousfield_Kan}, Goerss-Jardine \cite{Goerss_Jardine}, and Hovey \cite{Hovey}. We work in the category of symmetric spectra (see Hovey-Shipley-Smith \cite{Hovey_Shipley_Smith} and Schwede \cite{Schwede_book_project}), equipped with the injective stable model structure, so that ``$S$-modules'' means ``symmetric spectra'', which are the same as modules over the sphere spectrum $S$; in particular, fibrant $S$-modules enjoy the property of being $\Omega$-spectra \cite[1.4]{Hovey_Shipley_Smith} that are objectwise Kan complexes. For useful results and techniques involving spectra in closely related contexts, see Bousfield-Friedlander \cite{Bousfield_Friedlander}, Elmendorf-Kriz-Mandell-May \cite{EKMM}, and Jardine \cite{Jardine_generalized_etale}.

\subsection{The spaces-level stabilization map}

If $X$ is a pointed space, the stabilization map has the form
\begin{align}
\label{eq:hurewicz_map_introduction}
  \pi_*(X)\rarrow \pi_*^{\mathrm{s}}(X)=\colim\nolimits_r\pi_{*+r}(\Sigma^r X)
\end{align}
This comparison map between homotopy groups and stable homotopy groups comes from a spaces-level stabilization map of the form
\begin{align}
\label{eq:hurewicz_map_spaces_level_introduction}
  X\rightarrow\tilde{\Loop}^\infty\Susp^\infty(X)
\end{align}
and applying $\pi_*$ to \eqref{eq:hurewicz_map_spaces_level_introduction} recovers the map \eqref{eq:hurewicz_map_introduction}; here, $\tilde{\Loop}^\infty=\Loop^\infty F$ (Definition \ref{defn:fibrant_replacement}) denotes the right-derived functor of the underlying $0$-th space $\Loop^\infty=\Ev_0$ functor and $\Susp^\infty=S\tensor-$ denotes the stabilization functor given by tensoring with the sphere spectrum $S$.

\subsection{Iterating the stabilization map}

With a spaces-level stabilization map in hand, it is natural to form a cosimplicial resolution of $X$ with respect to $\tilde{\Loop}^\infty\Sigma^\infty$ of the form
\begin{align}
\label{eq:homology_resolution_introduction}
\xymatrix{
  X\ar[r] &
  \tilde{\Loop}^\infty\Sigma^\infty(X)\ar@<-0.5ex>[r]\ar@<0.5ex>[r] &
  (\tilde{\Loop}^\infty\Sigma^\infty)^2(X)
  \ar@<-1.0ex>[r]\ar[r]\ar@<1.0ex>[r] &
  (\tilde{\Loop}^\infty\Sigma^\infty)^3(X)\cdots
  }
\end{align}
showing only the coface maps. The homotopical comonad $\tilde{K}=\Susp^\infty\Loopt^\infty$, which is the derived functor of the comonad $K=\Susp^\infty\Loop^\infty$ associated to the $(\Susp^\infty, \Omega^\infty)$ adjunction, can be thought of as encoding the spectra-level co-operations on the suspension spectra; compare with \cite{Blomquist_Harper_integral_chains, Blomquist_Harper_iterated_suspension} for integral chains and iterated suspension.

By analogy with the techniques in Bousfield-Kan \cite{Bousfield_Kan}, by iterating the spaces-level stabilization map \eqref{eq:hurewicz_map_spaces_level_introduction} Carlsson \cite{Carlsson_equivariant}, and subsequently Arone-Kankaanrinta \cite{Arone_Kankaanrinta}, study the cosimplicial resolution of $X$ with respect to $\tilde{\Loop}^\infty\Sigma^\infty$, and taking the homotopy limit of the resolution \eqref{eq:homology_resolution_introduction} produce the $\tilde{\Loop}^\infty\Susp^\infty$-completion map
\begin{align}
\label{eq:comparison_map_Z_completion}
  X\rarrow X^\wedge_{\tilde{\Loop}^\infty\Susp^\infty}
\end{align}
We will show that this completion map fits into a derived equivalence with comparison map the stabilization $\Sigma^\infty$ functor (but landing in the naturally occurring category of coalgebraic objects).

\subsection{The main result}

In this paper we shall prove the following theorem, that the stabilization of spaces functor---the classical construction of associating a spectrum to a pointed space by tensoring with the sphere spectrum---satisfies homotopical descent on objects and morphisms. This is the stabilization analog of the Quillen-Sullivan theory main result \cite{Quillen_rational, Sullivan_infinitesimal} that the rational chains (resp. cochains) functor participates in a derived equivalence with certain coalgebra (resp. algebra) complexes, after restriction to 1-connected spaces up to rational equivalence. In more detail, we prove that the stabilization of spaces functor participates in a derived equivalence with certain coalgebra spectra (where the stabilization construction naturally lands), after restriction to 1-connected spaces up to weak equivalence. This resolves in the affirmative the infinite case, involving stabilization and suspension spectra, of a question/conjecture posed by Tyler Lawson on (iterated) suspension spaces almost ten years ago.

\begin{thm}
\label{MainTheorem}
The fundamental adjunction \eqref{eq:fundamental_adjunction_comparing_spaces_with_coAlgK} comparing pointed spaces to coalgebra spectra over the comonad $\K=\Susp^\infty\Loop^\infty$ via stabilization $\Susp^\infty$
\begin{align}
\label{eq:fundamental_adjunction_comparing_spaces_with_coAlgK}
\xymatrix{
  \Space_*\ar@<0.5ex>[r]^-{\Susp^\infty} & \coAlgK\ar@<0.5ex>[l]^-{\lim_\Delta C}
}
\end{align}
 induces a derived adjunction of the form
\begin{align}
\label{eq:derived_fundamental_adjunction_mapping_spaces}
  \Map_{\coAlgKt}(\Susp^\infty X,Y)\wequiv
  \Map_{\Space_*}(X,\holim\nolimits_\Delta \mathfrak{C}(Y))
\end{align}
that is a derived equivalence after restriction to the full subcategories of $1$-connected spaces and $1$-connected $\tilde{\K}$-coalgebra spectra; more precisely:
\begin{itemize}
\item[(a)] If $Y$ is a $1$-connected $\tilde{\K}$-coalgebra spectrum, then the derived counit map
\begin{align*}
  \Susp^\infty\holim\nolimits_\Delta \mathfrak{C}(Y)\xrightarrow{\wequiv} Y
\end{align*}
associated to the derived adjunction \eqref{eq:derived_fundamental_adjunction_mapping_spaces} is a weak equivalence; i.e., stabilization $\Susp^\infty$ is homotopically essentially surjective on $1$-connected coalgebra spectra over $\tilde{\K}$, and hence homotopical descent is satisfied on such objects.
\item[(b)] If $X'$ is a $1$-connected space, then the derived unit map
\begin{align*}
  X'\xrightarrow{\wequiv}\holim\nolimits_\Delta 
  \mathfrak{C}(\Susp^\infty X')
\end{align*} 
associated to the derived adjunction \eqref{eq:derived_fundamental_adjunction_mapping_spaces} is tautologically the $\tilde{\Loop}^\infty\Susp^\infty$-completion map $X'\rarrow {X'}^\wedge_{\tilde{\Loop}^\infty\Susp^\infty}$, and hence is a weak equivalence by \cite{Arone_Kankaanrinta, Carlsson}. In particular, the stabilization functor induces a weak equivalence
\begin{align}
  \label{eq:homotopically_fully_faithful_chains}
  \Susp^\infty\colon\thinspace\Map^h_{\Space_*}(X,X')\wequiv\Map_\coAlgKt(\Susp^\infty X,\Susp^\infty X')
\end{align}
on mapping spaces; i.e., stabilization $\Susp^\infty$ is homotopically fully faithful on $1$-connected spaces and hence homotopical descent is satisfied on such morphisms.
\end{itemize}
Here, $\Map^h_{\Space_*}(X,X')$ denotes the realization of the Dwyer-Kan \cite{Dwyer_Kan_function_complexes} homotopy function complex.
\end{thm}

See Arone-Ching \cite{Arone_Ching_classification} and Hess \cite{Hess} for a discussion of the homotopical descent ideas that have motivated and underlie much of our work, Carlsson-Milgram \cite{Carlsson_Milgram} and May \cite{May} for related ideas, Edwards-Hastings \cite{Edwards_Hastings} for a nice discussion of stabilization and abelianization in the spirit of this paper, and Behrens-Rezk \cite{Behrens_Rezk} for an interesting survey of closely related ideas circulating in the spectral algebra context.

\subsection{Classification and characterization theorems}
The following are corollaries of the main result.

\begin{thm}[Classification theorem]
\label{thm:classification}
A pair of $1$-connected pointed spaces $X$ and $X'$ are weakly equivalent if and only if the suspension spectra $\Susp^\infty X$ and $\Susp^\infty X'$ are weakly equivalent as derived $\Kt$-coalgebra spectra.
\end{thm}

\begin{thm}[Classification of maps theorem]
\label{thm:classification_maps}
Let $X,X'$ be pointed spaces. Assume that $X'$ is $1$-connected and fibrant.
\begin{itemize}
\item[(a)] \emph{(Existence)} Given any map $\phi$ in $[\Susp^\infty X,\Susp^\infty X']_\Kt$, there exists a map $f$ in $[X,X']$ such that $\phi=\Susp^\infty(f)$.
\item[(b)] \emph{(Uniqueness)} For each pair of maps $f,g$ in $[X,X']$, $f=g$ if and only if $\Susp^\infty(f)=\Susp^\infty(g)$ in the homotopy category of $\Kt$-coalgebra spectra.
\end{itemize}
\end{thm}

\begin{thm}[Characterization theorem]
\label{thm:characterization}
A $\Kt$-coalgebra spectrum $Y$ is weakly equivalent to the suspension spectrum $\Susp^\infty X$ of some $1$-connected space $X$, via derived $\Kt$-coalgebra maps, if and only if $Y$ is $1$-connected.
\end{thm}

\subsection{Strategy of attack and related work}

We are leveraging a line of attack developed in Ching-Harper \cite{Ching_Harper_derived_Koszul_duality} for resolving the $0$-connected case of the Francis-Gaitsgory conjecture \cite{Francis_Gaitsgory}, together with a powerful modification of that strategy developed in \cite{Blomquist_Harper_integral_chains} for resolving the integral chains problem, together with Cohn's \cite{Cohn} work showing that this strategy of attack extends to homotopy coalgebras over the associated homotopical comonad (see Blumberg-Riehl \cite{Blumberg_Riehl}); essentially, we establish and exploit uniform cartesian-ness estimates to show that the $\Loopt^\infty\Susp^\infty$-completion map studied in Carlsson \cite{Carlsson_equivariant}, and subsequently in Arone-Kankaanrinta \cite{Arone_Kankaanrinta}, participates in a derived equivalence between spaces and coalgebra spectra over the homotopical stabilization comonad, after restricting to $1$-connected spaces. A key ingredient underlying our homotopical estimates is the higher stabilization theorem proved in this paper; it can be thought of as the stabilization analog of Dundas' \cite{Dundas_relative_K_theory} higher Hurewicz theorem estimates for integral chains that play a key role in \cite{Blomquist_Harper_integral_chains} and underlies the main results in Dundas-Goodwillie-McCarthy \cite{Dundas_Goodwillie_McCarthy}.

The results of Hopkins \cite{Hopkins_iterated_suspension}, and the subsequent work of Goerss \cite{Goerss_desuspension} and Klein-Schwanzl-Vogt \cite{Klein_Schwanzl_Vogt} were motivating for us. This paper is a companion to our work in \cite{Blomquist_Harper_iterated_suspension} that resolves the original question/conjecture of Tyler Lawson \cite{Lawson_conjecture} on (iterated) suspension spaces; we regard this stabilization work as the stronger result: in the limit case we are throwing away more information, nevertheless, we find that homotopical descent on objects and morphisms still holds.

\subsection{Commuting stabilization with holim of a cobar construction}
Our main result reduces to proving our main technical theorem that the left derived stabilization functor $\Susp^\infty$ commutes,
\begin{align}
\label{eq:key_technical_result}
  \Susp^\infty\holim\nolimits_\Delta \mathfrak{C}(Y)\wequiv
  \holim\nolimits_\Delta \Susp^\infty \mathfrak{C}(Y)
\end{align}
up to weak equivalence, with the right derived limit functor $\holim_\Delta $, when composed with the cosimplicial cobar construction $\mathfrak{C}$ and evaluated on $1$-connected $\Kt$-coalgebra spectra.

\subsection{Organization of the paper}

In Section \ref{sec:outline_of_the_argument} we outline the argument of our main result. In Section \ref{sec:proofs} we review the fundamental adjunction and then prove the main result. Sections \ref{sec:simplicial_structures} through \ref{sec:derived_fundamental_adjunction}  are important background sections; for the convenience of the reader we briefly recall some preliminaries on simplicial structures, the homotopy theory of $\Kt$-coalgebras, and the derived fundamental adjunction for stabilization. For the experts familiar with Arone-Ching \cite{Arone_Ching_classification}, it will suffice to read Sections \ref{sec:outline_of_the_argument} and \ref{sec:proofs} for a complete proof of the main result.

\subsection*{Acknowledgments}

The authors would like to thank Michael Ching for helpful suggestions and useful remarks throughout this project. The second author would like to thank Bjorn Dundas, Bill Dwyer, Haynes Miller, and Crichton Ogle for useful suggestions and remarks and Lee Cohn, Mike Hopkins, Tyler Lawson, and Nath Rao for helpful comments. The second author is grateful to Haynes Miller for a stimulating and enjoyable visit to the Massachusetts Institute of Technology in early spring 2015, and to Bjorn Dundas for a stimulating and enjoyable visit to the University of Bergen in late spring 2015, and for their invitations which made this possible. The first author was supported in part by National Science Foundation grant DMS-1510640.

\section{Outline of the argument}
\label{sec:outline_of_the_argument}

In this section we will outline the proof of our main result. Since the derived unit map is tautologically the $\Loopt^\infty\Susp^\infty$-completion map $X'\rarrow {X'}^\wedge_{\Loopt^\infty\Susp^\infty}$, which is proved to be a weak equivalence on $1$-connected spaces in Carlsson \cite{Carlsson_equivariant} (see also Arone-Kankaanrinta \cite{Arone_Kankaanrinta}), proving the main result reduces to verifying that the derived counit map is a weak equivalence.

The following are proved in Section \ref{sec:cubical_diagrams_homotopical_analysis}.

\begin{thm}
\label{thm:estimating_connectivity_of_maps_in_tower_C_of_Y}
If $Y$ is a $1$-connected $\Kt$-coalgebra spectrum and $n\geq 1$, then the natural map
\begin{align}
\label{eq:tower_map_from_n_th_stage_to_next_lower_stage}
  \holim_{\Delta^{\leq n}}\mathfrak{C}(Y)&\longrightarrow
  \holim_{\Delta^{\leq n-1}}\mathfrak{C}(Y)
\end{align}
is an $(n+2)$-connected map between $1$-connected objects.
\end{thm}

\begin{thm}
\label{thm:connectivities_for_map_into_n_th_stage}
If $Y$ is a $1$-connected $\Kt$-coalgebra spectrum and $n\geq 0$, then the natural maps
\begin{align}
\label{eq:canonical_map_needed_to_discuss}
  \holim\nolimits_\Delta \mathfrak{C}(Y)&\longrightarrow
  \holim\nolimits_{\Delta^{\leq n}}\mathfrak{C}(Y)\\
\label{eq:the_other_canonical_map_needed}
  \Susp^\infty\holim\nolimits_\Delta \mathfrak{C}(Y)&\longrightarrow
  \Susp^\infty\holim\nolimits_{\Delta^{\leq n}}\mathfrak{C}(Y)
\end{align}
are $(n+3)$-connected maps between $1$-connected objects.
\end{thm}

\begin{proof}
Consider the first part. By Theorem \ref{thm:estimating_connectivity_of_maps_in_tower_C_of_Y} each of the maps in the holim tower $\{\holim_{\Delta^{\leq n}}\mathfrak{C}(Y)\}_n$, above level $n$, is at least $(n+3)$-connected. It follows that the map \eqref{eq:canonical_map_needed_to_discuss} is $(n+3)$-connected. The second part follows from the first part.
\end{proof}

The following is a key ingredient for proving our main result. It provides estimates sufficient for verifying that stabilization commutes past the desired homotopy limits.

\begin{thm}
\label{thm:connectivities_for_map_that_commutes_sigma_into_inside_of_holim}
If $Y$ is a $1$-connected $\Kt$-coalgebra spectrum and $n\geq 1$, then the natural map
\begin{align}
\label{eq:commuting_sigma_past_holim_delta}
  \Susp^\infty\holim\nolimits_{\Delta^{\leq n}} \mathfrak{C}(Y)\longrightarrow
  \holim\nolimits_{\Delta^{\leq n}} \Susp^\infty\,\mathfrak{C}(Y),
\end{align}
is $(n+5)$-connected; the map is a weak equivalence for $n=0$.
\end{thm}

The following resolves the original form of Tyler Lawson's question/conjecture.

\begin{thm}
\label{thm:sigmaC_commutes_with_desired_holim}
If $Y$ is a $1$-connected $\Kt$-coalgebra spectrum, then the natural maps
\begin{align}
\label{eq:sigmaC_commutes_with_desired_holim}
  \Susp^\infty\holim\nolimits_\Delta \mathfrak{C}(Y)&\xrightarrow{\wequiv}
  \holim\nolimits_\Delta \Susp^\infty\,\mathfrak{C}(Y)\xrightarrow{\wequiv}Y
\end{align}
are weak equivalences.
\end{thm}

\begin{proof}
For the case of the left-hand map, it is enough to verify that the connectivities of the natural maps \eqref{eq:the_other_canonical_map_needed} and \eqref{eq:commuting_sigma_past_holim_delta}
 are strictly increasing with $n$, and Theorems \ref{thm:connectivities_for_map_into_n_th_stage} and \ref{thm:connectivities_for_map_that_commutes_sigma_into_inside_of_holim} complete the proof. Consider the right-hand map. Since $\Susp^\infty\, \mathfrak{C}(Y)\wequiv F\Susp^\infty\, \mathfrak{C}(Y)$ and the latter is isomorphic to the cosimplicial cobar construction $\Cobar(F\K,F\K,FY)$, which has extra codegeneracy maps $s^{-1}$ (Dwyer-Miller-Neisendorfer \cite[6.2]{Dwyer_Miller_Neisendorfer}), it follows from the cofinality argument in Dror-Dwyer \cite[3.16]{Dror_Dwyer_long_homology} that the right-hand map in \eqref{eq:sigmaC_commutes_with_desired_holim} is a weak equivalence.
\end{proof}

\begin{proof}[Proof of Theorem \ref{MainTheorem}]
We want to verify that the natural map
$
  \Susp^\infty\holim\nolimits_\Delta \mathfrak{C}(Y)\rarrow Y
$
is a weak equivalence; since this is the composite \eqref{eq:sigmaC_commutes_with_desired_holim}, Theorem \ref{thm:sigmaC_commutes_with_desired_holim} completes the proof.
\end{proof}

\section{Homotopical analysis}
\label{sec:proofs}

In this section we recall the stabilization functor, together with related constructions, and prove Theorems \ref{thm:estimating_connectivity_of_maps_in_tower_C_of_Y} and \ref{thm:connectivities_for_map_that_commutes_sigma_into_inside_of_holim}. 

\subsection{Stabilization and the fundamental adjunction}

The fundamental adjunction naturally arises by observing that $\Susp^\infty$ is equipped with a coaction over the comonad $\K$ associated to the $(\Susp^\infty, \Omega^\infty)$ adjunction; this observation, which remains true for any adjunction provided that the indicated limits below exist, forms the basis of the homotopical descent ideas appearing in Hess \cite{Hess} and subsequently in Francis-Gaitsgory \cite{Francis_Gaitsgory}.

Consider any pointed space $X$ and $S$-module $Y$, and recall that the suspension spectrum $\Sigma^\infty(X)=S\tensor X$ and $0$-th space $\Loop^\infty(Y)=\Ev_0(Y)=Y_0$ functors fit into an adjunction
\begin{align}
\label{eq:suspension_adjunction}
\xymatrix{
  \Space_*\ar@<0.5ex>[r]^-{\Susp^\infty} &
  \Mod_S\ar@<0.5ex>[l]^-{\Loop^\infty}
}
\end{align}
with left adjoint on top. Associated to the adjunction in \eqref{eq:suspension_adjunction} is the monad $\Loop^\infty\Susp^\infty$ on pointed spaces $\Space_*$ and the comonad $\K:=\Susp^\infty\Loop^\infty$ on $S$-modules $\Mod_S$ of the form
\begin{align}
\label{eq:suspension_spectrum_functor_natural_transformations}
  \id\xrightarrow{\eta} \Loop^\infty\Susp^\infty&\quad\text{(unit)},\quad\quad
  &\id\xleftarrow{\varepsilon}\K& \quad\text{(counit)}, \\
  \notag
  \Loop^\infty\Susp^\infty\Loop^\infty\Susp^\infty\rarrow \Loop^\infty\Susp^\infty&\quad\text{(multiplication)},\quad\quad
  &\K\K\xleftarrow{m}\K& \quad\text{(comultiplication)}.
\end{align}
and it follows formally that there is a factorization of adjunctions of the form
\begin{align}
\label{eq:factorization_of_adjunctions_suspension}
\xymatrix{
  \Space_*\ar@<0.5ex>[r]^-{\Susp^\infty} &
  \coAlgK\ar@<0.5ex>[r]\ar@<0.5ex>[l]^-{\lim_\Delta C} &
  \Mod_S\ar@<0.5ex>[l]^-{\K} 
}
\end{align}
with left adjoints on top and $\coAlgK\rarrow\Mod_S$ the forgetful functor. In particular, the suspension spectrum $\Susp^\infty X$ is naturally equipped with a $\K$-coalgebra structure. To understand the comparison in \eqref{eq:factorization_of_adjunctions_suspension} between $\Space_*$ and $\coAlgK$ it is enough to observe that $\lim_\Delta C(Y)$ is naturally isomorphic to an equalizer of the form
\begin{align*}
  \lim_\Delta C(Y)\Iso
  \lim\Bigl(
  \xymatrix{
    \Loop^\infty Y\ar@<0.5ex>[r]^-{d^0}\ar@<-0.5ex>[r]_-{d^1} &
    \Loop^\infty\K Y
  }
  \Bigr)
\end{align*}
where $d^0=m\id$, $d^1=\id m$, $\function{m}{\Loop^\infty}{\Loop^\infty\K=\Loop^\infty\Susp^\infty\Loop^\infty}$ denotes the $\K$-coaction map on $\Loop^\infty$ (defined by $m:=\eta\id$), and $\function{m}{Y}{\K Y}$ denotes the $\K$-coaction map on $Y$; see, for instance, \cite{Blomquist_Harper_integral_chains}.

\begin{defn}
\label{defn:cobar_construction}
Let $Y$ be a $\K$-coalgebra. The \emph{cosimplicial cobar construction} $C(Y):=\Cobar(\Loop^\infty,\K,Y)$ in $(\Space_*)^{\Delta}$ looks like
\begin{align}
\label{eq:cobar_construction}
\xymatrix{
  C(Y): \quad\quad
  \Loop^\infty Y\ar@<0.5ex>[r]^-{d^0}\ar@<-0.5ex>[r]_-{d^1} &
  \Loop^\infty \K Y
  \ar@<1.0ex>[r]\ar[r]\ar@<-1.0ex>[r] &
  \Loop^\infty\K\K Y
  \cdots
}
\end{align}
(showing only the coface maps) and is defined objectwise by $C(Y)^n:=\Loop^\infty\K^n Y$ with the obvious coface and codegeneracy maps; see, for instance, the face and degeneracy maps in the simplicial bar constructions described in Gugenheim-May \cite[A.1]{Gugenheim_May} or May \cite[Section 7]{May_classifying_spaces}, and dualize. For instance, the indicated coface maps in \eqref{eq:cobar_construction} are defined by $d^0:=m\id$ and $d^1:=\id m$.
\end{defn}

\subsection{Coalgebras over the homotopical comonad $\Kt$}

It will be useful to interpret the cosimplicial $\Loopt^\infty\Susp^\infty$-resolution of $X$ in terms of a cosimplicial cobar construction that naturally arises as a ``fattened'' version of \eqref{eq:cobar_construction}; this leads to the notion of a $\Kt$-coalgebra exploited in Cohn \cite{Cohn}.

\begin{defn}
\label{defn:fibrant_replacement}
Denote by $\function{\eta}{\id}{F}$ and $\function{m}{FF}{F}$ the unit and multiplication maps of the simplicial fibrant replacement monad $F$ on $\Mod_S$ (Blumberg-Riehl \cite[6.1]{Blumberg_Riehl}). It follows that $\Loopt^\infty:=\Loop^\infty F$ and $\Kt:=\K F$ are the derived functors of $\Loop^\infty$ and $\K$, respectively. The comultiplication $\function{m}{\Kt}{\Kt\Kt}$ and counit $\function{\varepsilon}{\Kt}{F}$ maps are defined by the composites
\begin{align}
\label{eq:comultiplication_K_tilde}
  &\K F\xrightarrow{m\id}\K\K F=
  \K\id\K F\xrightarrow{\id\eta\id\id}
  \K F\K F\\
  \label{eq:counit_K_tilde}
  &\K F\xrightarrow{\varepsilon\id}\id F=F
\end{align}
respectively.
\end{defn}

It is shown in Blumberg-Riehl \cite[4.2, 4.4]{Blumberg_Riehl}, and subsequently exploited in Cohn \cite{Cohn}, that the derived functor $\Kt:=\K F$ of the comonad $\K$ is very nearly a comonad itself with the structure maps $\function{m}{\Kt}{\Kt\Kt}$ and $\function{\varepsilon}{\Kt}{F}$ above. For instance, it is proved in \cite{Blumberg_Riehl} that $\Kt$ defines a comonad on the homotopy category of $\Mod_S$, which is a reflection of the the fact that $\Kt$ has the structure of a highly homotopy coherent comonad (see \cite{Blumberg_Riehl}); in particular, $\Kt$ has a strictly coassociative comultiplication $\function{m}{\Kt}{\Kt\Kt}$ and satisfies left and right counit identities up to factors of $F\wequiv\id$. In more detail, the homotopical comonad $\Kt$ makes the following diagrams
\begin{align*}
  \xymatrix{
    \Kt\ar[r]^-{m}\ar[d]_-{m} & 
    \Kt\Kt\ar[d]^-{m\id}\\
    \Kt\Kt\ar[r]_-{\id m} & \Kt\Kt\Kt
  }\quad\quad\quad
  \xymatrix{
    F\Kt\ar[r]^-{\id m}\ar@{=}[d] & F\Kt \Kt\ar[d]^-{(*)}\\
     F\Kt\ar@{=}[r] & F\Kt
  }\quad\quad\quad
  \xymatrix{
    \Kt\ar[r]^-{m}\ar@{=}[d] &
    \Kt\Kt\ar[d]^-{(**)}\\
    \Kt\ar@{=}[r] & \Kt
  }
\end{align*}
commute; here, the map $(*)$ is the composite
$
  F\Kt \Kt\xrightarrow{\id\varepsilon\id}FF\Kt\xrightarrow{m\id}F\Kt
$ and the map $(**)$ is the composite
$
  \K F\Kt\xrightarrow{\id\id\varepsilon}\K FF\xrightarrow{\id m}\K F
$.

\begin{rem}
Associated to the adjunction $(\Susp^\infty,\Loop^\infty)$ is a left $\K$-coaction (or $\K$-coalgebra structure) $\function{m}{\Susp^\infty X}{\K \Susp^\infty X}$ on $\Susp^\infty X$, defined by $m=\id\eta\id$), for any $X\in\Space_*$. This map induces a corresponding left $\Kt$-coaction $\function{m}{\Susp^\infty X}{\Kt\Susp^\infty X}$ that is the composite 
\begin{align*}
  \Susp^\infty X\xrightarrow{m}
  \K\Susp^\infty X=\K\id\Susp^\infty X
  \xrightarrow{}\K F\Susp^\infty X
\end{align*}
\end{rem}

The following notion of a homotopy $\Kt$-coalgebra, exploited in Cohn \cite{Cohn}, captures exactly the left $\Kt$-coaction structure that stabilization $\Susp^\infty X$ of a pointed space $X$ satisfies; this is precisely the structure being encoded by the cosimplicial $\Loopt^\infty\Sigma^\infty$ resolution \eqref{eq:homology_resolution_introduction}. 
 
\begin{defn}
A \emph{homotopy $\Kt$-coalgebra} (or $\Kt$-coalgebra, for short) is a $Y\in\Mod_S$ together with a map $\function{m}{Y}{\Kt Y}$ in $\Mod_S$ such that the following diagrams
\begin{align*}
  \xymatrix{
  Y\ar[r]^-{m}\ar[d]_-{m} & \Kt Y\ar[d]^-{m\id}\\
  \Kt Y\ar[r]_-{\id m} & \Kt\Kt Y
  }\quad\quad\quad
  \xymatrix{
    FY\ar[r]^-{\id m}\ar@{=}[d] & F\Kt Y\ar[d]^-{(*)}\\
     FY\ar@{=}[r] & FY
  }
\end{align*}
commute; here, the map $(*)$ is the composite
$
  F\Kt Y\xrightarrow{\id\varepsilon\id}FFY\xrightarrow{m\id}FY
$.
\end{defn}

\begin{defn}
\label{defn:cobar_construction_fattened}
Let $Y$ be a $\Kt$-coalgebra. The \emph{cosimplicial cobar construction} $\mathfrak{C}(Y):=\Cobar(\Loopt^\infty,\Kt,Y)$ in $(\Space_*)^{\Delta}$ looks like
\begin{align}
\label{eq:cobar_construction_fattened_up}
\xymatrix{
  \mathfrak{C}(Y): \quad\quad
  \Loopt^\infty Y\ar@<0.5ex>[r]^-{d^0}\ar@<-0.5ex>[r]_-{d^1} &
  \Loopt^\infty\Kt Y
  \ar@<1.0ex>[r]\ar[r]\ar@<-1.0ex>[r] &
  \Loopt^\infty\Kt\Kt Y
  \cdots
}
\end{align}
(showing only the coface maps) and is defined objectwise by $\mathfrak{C}(Y)^n:=\Loopt^\infty\Kt^n Y=\Loop^\infty F(\K F)^n Y$ with the obvious coface and codegeneracy maps; for instance, in \eqref{eq:cobar_construction_fattened_up} the indicated coface maps are defined by $d^0:=m\id$ and $d^1:=\id m$ (compare with \eqref{eq:cobar_construction}).
\end{defn}

It is important to note that the cosimplicial resolution \eqref{eq:homology_resolution_introduction} of a pointed space $X$ with respect to stabilization $\Loopt^\infty\Susp^\infty$, built by iterating the spaces-level stabilization map \eqref{eq:hurewicz_map_spaces_level_introduction}, is naturally isomorphic to the map $X\rightarrow\mathfrak{C}(\Susp^\infty X)$; in other words, the homotopical comonad $\Kt$ can be thought of as encoding the spectra-level co-operations on the suspension spectra.

\begin{rem}
It may be helpful to note, in particular, that the cosimplicial cobar construction is encoding the fact that the derived functor $\Loopt^\infty$ has a naturally occurring right $\Kt$-coaction map $\function{m}{\Loopt^\infty}{\Loopt^\infty\Kt}$ that makes the following diagrams
\begin{align*}
  \xymatrix{
    \Loopt^\infty\ar[r]^-{m}\ar[d]_-{m} & 
    \Loopt^\infty\Kt\ar[d]^-{m\id}\\
    \Loopt^\infty\Kt\ar[r]_-{\id m} & \Loopt^\infty\Kt\Kt
  }\quad\quad\quad
  \xymatrix{
    \Loopt^\infty\ar[r]^-{m}\ar@{=}[d] &
    \Loopt^\infty\Kt\ar[d]^-{(**)}\\
    \Loopt^\infty\ar@{=}[r] & \Loopt^\infty
  }
\end{align*}
commute; here, the map $(**)$ is the composite
$
  \Loop^\infty F\Kt\xrightarrow{\id\id\varepsilon}\Loop^\infty FF\xrightarrow{\id m}\Loop^\infty F
$.
\end{rem}

\begin{rem}
It may be helpful to note that the counit map \eqref{eq:counit_K_tilde} is identical to the composite
\begin{align*}
  \K F=\id\K F\xrightarrow{\eta\id\id} F\K F\xrightarrow{\id\varepsilon\id} F\id F=FF\xrightarrow{m} F
\end{align*}
when comparing with \cite{Blumberg_Riehl}.
\end{rem}

\subsection{Higher stabilization}

The purpose of this section is to prove Theorem \ref{thm:higher_stabilization}; see Ching-Harper \cite{Ching_Harper} and Goodwillie \cite{Goodwillie_calculus_2} for notations and definitions associated to cubical diagrams.

\begin{defn}
Let $\function{f}{\NN}{\NN}$ be a function and $W$ a finite set. A $W$-cube $\capX$ is \emph{$f$-cartesian} (resp. \emph{$f$-cocartesian}) if each $d$-subcube of $\capX$ is $f(d)$-cartesian (resp. $f(d)$-cocartesian); here, $\NN$ denotes the non-negative integers.
\end{defn}

The following proposition, proved in \cite{Blomquist_Harper_iterated_suspension}, will be helpful in organizing the proof of the higher stabilization theorem below; compare with Dundas-Goodwillie-McCarthy \cite[A.8.3]{Dundas_Goodwillie_McCarthy}.

\begin{prop}[Uniformity correspondence]
\label{prop:uniformity_correspondence}
Let $k\geq 1$ and $W$ a finite set. A $W$-cube of pointed spaces is $(k(\id+1)+1)$-cartesian if and only if it is $((k+1)(\id+1)-1)$-cocartesian.
\end{prop}

The following theorem plays a key role in our homotopical analysis of the derived counit map below; it also provides an alternate proof, with stronger estimates, of the result in \cite{Arone_Kankaanrinta, Carlsson_equivariant} that the $\Loopt^\infty\Susp^\infty$-completion map $X\rarrow X^\wedge_{\Loopt^\infty\Susp^\infty}$ is a weak equivalence for any $1$-connected space $X$. It is motivated by Dundas \cite[2.6]{Dundas_relative_K_theory} and is the infinite or limit case of the closely related higher Freudenthal suspension theorem \cite{Blomquist_Harper_iterated_suspension}.

\begin{thm}[Higher stabilization theorem]
\label{thm:higher_stabilization}
Let $k\geq 1$, $W$ a finite set, and $\capX$ a $W$-cube of pointed spaces. If $\capX$ is $(k(\id+1)+1)$-cartesian, then so is $\capX\rarrow\tilde{\Loop}^\infty\Susp^\infty\capX$.
\end{thm}

\begin{proof}
Consider the case $|W|=0$. Suppose $\capX$ is a $W$-cube and $\capX_\emptyset$ is $k$-connected. We know by Freudenthal suspension, which can be understood as a consequence of the Blakers-Massey theorem (see, for instance, \cite[A.8.2]{Dundas_Goodwillie_McCarthy}), that the map $\capX_\emptyset\rarrow\tilde{\Loop}\Susp\capX_\emptyset$ is $(2k+1)$-connected. More generally, it follows by repeated application of Freudenthal suspension that the map $\capX_\emptyset\rarrow\tilde{\Loop}^\infty\Susp^\infty\capX_\emptyset$ is a $(2k+1)$-connected map between $k$-connected spaces. 

Consider the case $|W|\geq 1$. Suppose $\capX$ is a $W$-cube and $\capX$ is $(k(\id+1)+1)$-cartesian. Let's verify that $\capX\rarrow\tilde{\Loop}^\infty\Susp^\infty\capX$ is $(k(\id+1)+1)$-cartesian $(|W|+1)$-cube. It suffices to assume that $\capX$ is a cofibration $W$-cube; see \cite[1.13]{Goodwillie_calculus_2} and \cite[3.4]{Ching_Harper}. Let $C$ be the iterated cofiber of $\capX$ and $\capC$ the $W$-cube defined objectwise by $\capC_V={*}$ for $V\neq W$ and $\capC_W=C$. Then $\capX\rarrow\capC$ is $\infty$-cocartesian. Consider the commutative diagram
\begin{align}
\label{eq:diagram_of_cubes_for_higher_freudenthal}
\xymatrix{
  \capX\ar[r]\ar[d]_(0.4){(*)} & 
  \capC\ar[d]\\
  \tilde{\Loop}^\infty\Susp^\infty\capX\ar[r] & 
  \tilde{\Loop}^\infty\Susp^\infty\capC
}
\end{align}
of $|W|$-cubes. 

Let's verify that $(*)$ is $(k(|W|+2)+1)$-cartesian as a $(|W|+1)$-cube of pointed spaces. We know that $\capX$ is $((k+1)(\id+1)-1)$-cocartesian by Proposition \ref{prop:uniformity_correspondence}, and in particular, $C$ is $((k+1)(|W|+1)-1)$-connected. For $d<|W|$, any $(d+1)$ dimensional subcube of $\capX$ is $((k+1)(d+2)-1)=((k+1)(d+1)+k)$-cocartesian and any $d$ dimensional subcube of $\capX$ is $((k+1)(d+1)-1)$-cocartesian.  So if $\capX|T$ is some $d$-subcube of $\capX$ with $T$ not containing the terminal set $W$, then $\capX|T\rarrow\capC|T={*}$ is $(k+1)(d+1)$-cocartesian by \cite[1.7]{Goodwillie_calculus_2}. Furthermore, even if $T$ contains the terminal set $W$, we know that $\capX|T\rarrow\capC|T$ is still $(k+1)(d+1)$-cocartesian by \cite[1.7]{Goodwillie_calculus_2}; this is because $(k+1)(d+1)<(k+1)(|W|+1)-1$ since $k\geq 1$ and $d<|W|$. Hence $\capX|T\rarrow\capC|T$ is $(k+1)(d+1)$-cocartesian for any $d$-subcube $\capX|T$ of $\capX$. It follows easily from higher Blakers-Massey \cite[2.5]{Goodwillie_calculus_2} that $\capX\rarrow\capC$ is $(k(|W|+2)+1)$-cartesian. 

We know that $\Susp^\infty\capX\rarrow\Susp^\infty\capC$ is $\infty$-cocartesian and hence $\infty$-cartesian; therefore $\tilde{\Loop}^\infty\Susp^\infty\capX\rarrow\tilde{\Loop}^\infty\Susp^\infty\capC$ is $\infty$-cartesian. Also, $\capC\rarrow\tilde{\Loop}^\infty\Susp^\infty\capC$ is at least $(k(|W|+2)+1)$-cartesian since $C\rarrow\tilde{\Loop}^\infty\Susp^\infty C$ is $(2[(k+1)(|W|+1)-1]+1)$-connected by Freudenthal suspension; this is because the cartesian-ness of $\capC\rarrow\tilde{\Loop}^\infty\Susp^\infty\capC$ is the same as the connectivity of the map $\tilde{\Loop}^{|W|}\capC\rarrow\tilde{\Loop}^{|W|}\tilde{\Loop}^\infty\Susp^\infty\capC$ (by considering iterated homotopy fibers).

Putting it all together, it follows from diagram \eqref{eq:diagram_of_cubes_for_higher_freudenthal} and \cite[1.8]{Goodwillie_calculus_2} that the map $(*)$ is $(k(|W|+2)+1)$-cartesian. Doing this also on all subcubes gives the result.
\end{proof}

\subsection{Homotopical estimates and codegeneracy cubes}

\label{sec:cubical_diagrams_homotopical_analysis}

Here we prove Theorems \ref{thm:estimating_connectivity_of_maps_in_tower_C_of_Y} and \ref{thm:connectivities_for_map_that_commutes_sigma_into_inside_of_holim}. The following calculates the layers of the Tot tower; see, for instance, Bousfield-Kan \cite[X.6.3]{Bousfield_Kan}.

\begin{prop}
\label{prop:iterated_homotopy_fibers_calculation}
Let $Z$ be a cosimplicial pointed space and $n\geq 0$. There are natural zigzags of weak equivalences
\begin{align*}
  \hofib(\holim_{\Delta^{\leq n}}Z\rarrow\holim_{\Delta^{\leq n-1}}Z)
  \wequiv\Loop^n(\iter\hofib)\capY_n
\end{align*}
where $\capY_n$ denotes the canonical $n$-cube built from the codegeneracy maps of
\begin{align*}
\xymatrix{
  Z^0 &
  Z^1
  \ar[l]_-{s^0} &
  Z^2\ar@<-0.5ex>[l]_-{s^0}\ar@<0.5ex>[l]^-{s^1}
  \ \cdots\ Z^n
}
\end{align*}
 the $n$-truncation of $Z$; in particular, $\capY_0$ is the object (or $0$-cube) $Z^0$. We often refer to $\capY_n$ as the \emph{codegeneracy} $n$-cube associated to $Z$.
\end{prop}

The following is proved in Carlsson \cite[Section 6]{Carlsson}, Dugger \cite{Dugger_homotopy_colimits}, and Sinha \cite[6.7]{Sinha_cosimplicial_models}, and plays a key role in this paper; see also  Dundas-Goodwillie-McCarthy \cite{Dundas_Goodwillie_McCarthy} and Munson-Volic \cite{Munson_Volic_book_project}; it was exploited early on by Hopkins \cite{Hopkins_iterated_suspension}.

\begin{prop}
\label{prop:left_cofinality_truncated_delta}
Let $n\geq 0$. The composite
\begin{align*}
  \capP_0([n])\Iso P\Delta[n]\longrightarrow\Delta_\res^{\leq n}
  \subset\Delta^{\leq n}
\end{align*}
is left cofinal (i.e., homotopy initial). Here, $\capP_0([n])$ denotes the poset of all nonempty subsets of $[n]$ and $P\Delta[n]$ denotes the poset of non-degenerate simplices of the standard $n$-simplex $\Delta[n]$; see \cite[III.4]{Goerss_Jardine}.

\end{prop}

\begin{prop}
\label{prop:punctured_cube_calculation_of_holim_truncated_delta}
If $X\in\M^\Delta$ is objectwise fibrant, then the natural maps
\begin{align*}
  \holim\nolimits_{\Delta^{\leq n}}^\BK X&\xrightarrow{\wequiv}
  \holim\nolimits_{P\Delta[n]}^\BK X\Iso
  \holim\nolimits_{\capP_0([n])}^\BK X
\end{align*}
in $\M$ are weak equivalences; here, $\M$ is any simplicial model category.
\end{prop}

\begin{rem}
We follow the conventions and definitions in Bousfield-Kan \cite{Bousfield_Kan}, together with \cite{Blomquist_Harper_integral_chains} and Ching-Harper \cite{Ching_Harper_derived_Koszul_duality} for the various models of homotopy limits.
\end{rem}

\begin{defn}
\label{defn:the_wide_tilde_construction}
Let $Z$ be a cosimplicial pointed space and $n\geq 0$. Assume that $Z$ is objectwise fibrant and denote by $\function{Z}{\capP_0([n])}{\Space_*}$ the composite
\begin{align*}
  \capP_0([n])\rightarrow\Delta^{\leq n}
  \rightarrow\Delta
  \rightarrow\Space_*
\end{align*}
The \emph{associated $\infty$-cartesian $(n+1)$-cube built from $Z$}, denoted $\function{\widetilde{Z}}{\capP([n])}{\Space_*}$, is defined objectwise by
\begin{align*}
  \widetilde{Z}_V :=
  \left\{
    \begin{array}{rl}
    \holim^\BK_{T\neq\emptyset}Z_T,&\text{for $V=\emptyset$,}\\
    Z_V,&\text{for $V\neq\emptyset$}.
    \end{array}
  \right.
\end{align*}
\end{defn}

\begin{prop}[Uniformity of faces]
\label{prop:comparing_faces_of_coface_cube_with_codegeneracy_cube}
Let $Z\in(\Space_*)^\Delta$ and $n\geq 0$. Assume that $Z$ is objectwise fibrant. Let $\emptyset\neq T\subset[n]$ and $t\in T$. Then there is a weak equivalence
\begin{align*}
  (\iter\hofib)\partial_{\{t\}}^T\widetilde{Z}\wequiv
  \Omega^{|T|-1}(\iter\hofib)\capY_{|T|-1}
\end{align*}
in $\Space_*$, where $\capY_{|T|-1}$ denotes the codegeneracy $(|T|-1)$-cube associated to $Z$.
\end{prop}

\begin{proof}
This is proved in \cite{Ching_Harper_derived_Koszul_duality}; compare Goodwillie \cite[3.4]{Goodwillie_calculus_3} and Sinha \cite[7.2]{Sinha_cosimplicial_models}.
\end{proof}

\begin{thm}
\label{thm:cocartesian_and_cartesian_estimates}
Let $Y$ be a $\Kt$-coalgebra and $n\geq 1$. Consider the $\infty$-cartesian $(n+1)$-cube $\widetilde{\mathfrak{C}(Y)}$ in $\Space_*$ built from $\mathfrak{C}(Y)$. If $Y$ is $1$-connected, then
\begin{itemize}
\item[(a)] the cube $\widetilde{\mathfrak{C}(Y)}$ is $(2n+5)$-cocartesian in $\Space_*$,
\item[(b)] the cube $\Susp^\infty\widetilde{\mathfrak{C}(Y)}$ is $(2n+5)$-cocartesian in $\Mod_S$,
\item[(c)] the cube $\Susp^\infty\widetilde{\mathfrak{C}(Y)}$ is $(n+5)$-cartesian in $\Mod_S$.
\end{itemize}
\end{thm}

\begin{proof}
Consider part (a) and let $W=[n]$. Our strategy is to use the higher dual Blakers-Massey theorem in Goodwillie \cite[2.6]{Goodwillie_calculus_2} to estimate how close the $W$-cube $\widetilde{\mathfrak{C}(Y)}$ in $\Space_*$ is to being cocartesian. We know from the higher stabilization theorem, Theorem  \ref{thm:higher_stabilization}, on iterations of the stabilization map applied to $\Loopt^\infty Y$, together with the uniformity enforced by Proposition \ref{prop:comparing_faces_of_coface_cube_with_codegeneracy_cube}, that for each nonempty subset $V\subset W$, the $V$-cube $\partial_{W-V}^W\widetilde{\mathfrak{C}(Y)}$ is $(|V|+2)$-cartesian; since it is $\infty$-cartesian by construction when $V=W$, it follows immediately from Goodwillie \cite[2.6]{Goodwillie_calculus_2} that $\widetilde{\mathfrak{C}(Y)}$ is $(2n+5)$-cocartesian in $\Space_*$, which finishes the proof of part (a). Part (b) follows from the fact that $\function{\Susp^\infty}{\Space_*}{\Mod_S}$ preserves cocartesian-ness. Part (c) follows easily from \cite[3.10]{Ching_Harper}.
\end{proof}

\begin{proof}[Proof of Theorem \ref{thm:connectivities_for_map_that_commutes_sigma_into_inside_of_holim}]
We want to estimate how connected the comparison map
\begin{align*}
  \Susp^\infty\holim\nolimits_{\Delta^{\leq n}} \mathfrak{C}(Y)\longrightarrow
  \holim\nolimits_{\Delta^{\leq n}} \Susp^\infty\,\mathfrak{C}(Y),
\end{align*}
is, which is equivalent to estimating how cartesian $\Susp^\infty\widetilde{\mathfrak{C}(Y)}$ is; Theorem \ref{thm:cocartesian_and_cartesian_estimates}(c) completes the proof.
\end{proof}

\begin{prop}
\label{prop:iterated_hofiber_codegeneracy_cube}
Let $Y$ be a $\Kt$-coalgebra and $n\geq 1$. Denote by $\capY_n$ the codegeneracy $n$-cube associated to the cosimplicial cobar construction $\mathfrak{C}(Y)$ of $Y$. If $Y$ is $1$-connected, then the total homotopy fiber of $\capY_n$ is $(2n+1)$-connected.
\end{prop}

\begin{proof}
This follows immediately from the proof of Theorem \ref{thm:cocartesian_and_cartesian_estimates}, together with Proposition \ref{prop:comparing_faces_of_coface_cube_with_codegeneracy_cube}.
\end{proof}

\begin{proof}[Proof of Theorem \ref{thm:estimating_connectivity_of_maps_in_tower_C_of_Y}]
The homotopy fiber of the map \eqref{eq:tower_map_from_n_th_stage_to_next_lower_stage} is weakly equivalent to $\Loopt^n$ of the total homotopy fiber of the codegeneracy $n$-cube $\capY_{n}$ associated to $\mathfrak{C}(Y)$ by Proposition \ref{prop:iterated_homotopy_fibers_calculation}, hence by Proposition \ref{prop:iterated_hofiber_codegeneracy_cube} the map \eqref{eq:tower_map_from_n_th_stage_to_next_lower_stage} is $(n+2)$-connected.
\end{proof}

As an immediate corollary of Theorem \ref{thm:estimating_connectivity_of_maps_in_tower_C_of_Y}, we get the following.

\begin{thm}
If $Y$ is a $1$-connected $\Kt$-coalgebra spectrum, then the homotopy spectral sequence
\begin{align*}
  E^2_{-s,t} &= \pi^s\pi_t \mathfrak{C}(Y)
  \Longrightarrow
  \pi_{t-s}\holim\nolimits_\Delta \mathfrak{C}(Y)
\end{align*}
converges strongly; compare with Ching-Harper \cite{Ching_Harper_derived_Koszul_duality}.
\end{thm}

\begin{proof}
This follows from the connectivity estimates in Theorem \ref{thm:estimating_connectivity_of_maps_in_tower_C_of_Y}.
\end{proof}

These types of homotopy spectral sequences have been studied in Bousfield-Kan \cite{Bousfield_Kan_spectral_sequence}, Bendersky-Curtis-Miller \cite{Bendersky_Curtis_Miller} and Bendersky-Thompson \cite{Bendersky_Thompson}.

\section{Background on simplicial structures}
\label{sec:simplicial_structures}

It will be useful to recall, in this background section, the simplicial structures on pointed spaces and $S$-modules.

\begin{defn}
\label{defn:simplicial_structure_pointed_spaces}
Let $X,X'$ be pointed spaces and $K$ a simplicial set. The \emph{tensor product} $X\tensordot K$ in $\Space_*$, \emph{mapping object} $\hombold_{\Space_*}(K,X)$ in $\Space_*$, and \emph{mapping space} $\Hombold_{\Space_*}(X,X')$ in $\sSet$ are defined by
\begin{align*}
  X\tensordot K &:=X\Smash K_+\\
  \hombold_{\Space_*}(K,X')&:=\hombold_*(K_+,X')\\
  \Hombold_{\Space_*}(X,X')_n &:= \hom_{\Space_*}(X\tensordot\Delta[n],X')
\end{align*}
where the \emph{pointed mapping space} $\hombold_*(X,X')$ in $\Space_*$ is $\Hombold_{\Space_*}(X,X')$ pointed by the constant map; see \cite[II.3]{Goerss_Jardine}.
\end{defn}

\begin{defn}
\label{defn:simplicial_structure_ModS}
Let $Y,Y'$ be $S$-modules and $K$ a simplicial set. The \emph{tensor product} $Y\tensordot K$ in $\Mod_S$, \emph{mapping object} $\hombold_{\Mod_S}(K,Y)$ in $\Mod_S$, and \emph{mapping space} $\Hombold_{\Mod_S}(Y,Y')$ in $\sSet$ are defined by
\begin{align*}
  Y\tensordot K &:=Y\Smash K_+\\
  \hombold_{\Mod_S}(K,Y')&:=\Map(K_+,Y')\\
  \Hombold_{\Mod_S}(Y,Y')_n &:= \hom_{\Mod_S}(Y\tensordot\Delta[n],Y')
\end{align*}
where $\Map(K_+,Y')$ denotes the function $S$-module; see \cite[2.2.9]{Hovey_Shipley_Smith}. 
\end{defn}

We sometimes drop the $\Space_*$ and $\Mod_S$ decorations from the notation and simply write $\Hombold$ and $\hombold$.

\begin{prop}
With the above definitions of mapping object, tensor product, and mapping space the categories of pointed spaces $\Space_*$ and $S$-modules $\Mod_S$  are simplicial model categories.
\end{prop}

\begin{proof}
This is proved, for instance, in \cite[II.3]{Goerss_Jardine} and \cite{Hovey_Shipley_Smith}.
\end{proof}

Recall that the stabilization adjunction $(\Susp^\infty, \Loop^\infty)$ in \eqref{eq:suspension_adjunction} is a Quillen adjunction with left adjoint on top; in particular, for $X,Y\in\Space_*$ there is an isomorphism
\begin{align}
\label{eq:hom_set_adjunction_stabilization_general}
  \hom(\Susp^\infty X,Y)\Iso\hom(X,\Loop^\infty Y)
\end{align}
in $\Set$, natural in $X,Y$.

The following proposition, which follows from Goerss-Jardine \cite[II.2.9]{Goerss_Jardine}, verifies that the stabilization adjunction \eqref{eq:suspension_adjunction} meshes nicely with the simplicial structure.

\begin{prop}
\label{prop:useful_properties_of_the_adjunction}
Let $X$ be a pointed space, $Y$ an $S$-module, and $K,L$ simplicial sets.  Then
\begin{itemize}
\item[(a)] there is a natural isomorphism
$
  \sigma\colon\thinspace \Susp^\infty(X)\tensordot K \xrightarrow{\Iso}\Susp^\infty(X\tensordot K)
$;
\item[(b)] there is an isomorphism
\begin{align*}
  \Hombold(\Susp^\infty X,Y)\Iso\Hombold(X,\Loop^\infty Y)
\end{align*}
in $\sSet$, natural in $X,Y$, that extends the adjunction isomorphism in  \eqref{eq:hom_set_adjunction_stabilization_general};
\item[(c)] there is an isomorphism
\begin{align*}
  \Loop^\infty\hombold(K,Y)\Iso\hombold(K,\Loop^\infty Y)
\end{align*}
in $\Space_*$, natural in $K,Y$.
\item[(d)] there is a natural map
$
  \function{\sigma}{\Loop^\infty(Y)\tensordot K}{\Loop^\infty(Y\tensordot K)}
$
induced by $\Loop^\infty$.
\item[(e)] the functors $\Susp^\infty$ and $\Loop^\infty$ are simplicial functors (Remark \ref{rem:simplicial_functors}) with the structure maps $\sigma$ of (a) and (d), respectively.
\end{itemize}
\end{prop}

\begin{rem}
\label{rem:simplicial_functors}
For a useful reference on simplicial functors in the context of homotopy theory, see Hirschhorn \cite[9.8.5]{Hirschhorn}.
\end{rem}

The following proposition is fundamental to this paper.

\begin{prop}
\label{prop:unit_and_counit_are_simplicial}
Consider the monad $\Loop^\infty\Susp^\infty$ on pointed spaces $\Space_*$ and the comonad $\Susp^\infty\Loop^\infty$ on $S$-modules $\Mod_S$ associated to the adjunction $(\Susp^\infty,\Loop^\infty)$ in \eqref{eq:suspension_adjunction}. The associated natural transformations
\begin{align*}
  \id\xrightarrow{\eta} \Loop^\infty\Susp^\infty,\quad\quad\quad
  &\id\xleftarrow{\varepsilon}\Susp^\infty\Loop^\infty\\
  \Loop^\infty\Susp^\infty\Loop^\infty\Susp^\infty\rarrow \Loop^\infty\Susp^\infty\quad\quad\quad
  &\Susp^\infty\Loop^\infty\Susp^\infty\Loop^\infty\xleftarrow{m}\Susp^\infty\Loop^\infty
\end{align*}
are simplicial natural transformations.
\end{prop}

\begin{proof}
This is an exercise left to the reader; compare \cite[Proof of 3.16]{Ching_Harper_derived_Koszul_duality}.
\end{proof}

\section{Background on the homotopy theory of $\Kt$-coalgebras}
\label{sec:homotopy_theory_K_coalgebras}

In this section we recall the homotopy theory of $\Kt$-coalgebras developed in Arone-Ching \cite{Arone_Ching_classification}; in fact, we use a tiny modification exploited in Ching-Harper \cite{Ching_Harper} and Cohn \cite{Cohn}. The expert already familiar with  \cite{Arone_Ching_classification} may wish to skip this background section.
 
A morphism of $\Kt$-coalgebra spectra from $Y$ to $Y'$ is a map $\function{f}{Y}{Y'}$ in $\Mod_S$ that makes the diagram
\begin{align}
\label{eq:commutative_square_defining_K_coalgebra_map}
\xymatrix{
  Y\ar[d]_-{f}\ar[r]^-{m} & \Kt Y\ar[d]^-{\id f}\\
  Y'\ar[r]_-{m} & \Kt Y'
}
\end{align}
in $\Mod_S$ commute. This motivates the following homotopically meaningful cosimplicial resolution of $\Kt$-coalgebra maps.

\begin{defn}
\label{defn:derived_K_coalgebra_maps}
Let $Y,Y'$ be $\Kt$-coalgebra spectra. The diagram $\Hombold\bigl(Y,F\Kt^\bullet Y'\bigr)$ in $(\sSet)^{\Delta}$ looks like
\begin{align*}
\xymatrix{
  \Hombold(Y,FY')\ar@<0.5ex>[r]^-{d^0}\ar@<-0.5ex>[r]_-{d^1} &
  \Hombold\bigl(Y,F\Kt Y'\bigr)
  \ar@<1.0ex>[r]\ar[r]\ar@<-1.0ex>[r] &
  \Hombold\bigl(Y,F\Kt\Kt Y'\bigr)\cdots
}
\end{align*}
and is defined objectwise by 
\begin{align*}
  \Hombold\bigl(Y,F\Kt^\bullet Y'\bigr)^n:=
  \Hombold\bigl(Y,F\Kt^n Y'\bigr)=
  \Hombold\bigl(Y,F(\K F)^n Y'\bigr)
\end{align*}
 with the obvious coface and codegeneracy maps.
\end{defn}

\begin{rem}
This is simply the resolution in Arone-Ching \cite{Arone_Ching_classification}, but ``fattened-up'' by $F$. For instance, on the level of hom-sets (simplicial degree 0), let's verify that $s^0d^1=
\id$ on $\Hombold(Y,FY')$. Start with $\function{f}{Y}{FY'}$ and consider the commutative diagram
\begin{align*}
\xymatrix{
  Y\ar[r]^-{f}\ar@{=}[dd] & 
  FY'\ar[r]^-{\id m}\ar@{=}[dd] & 
  F\K FY'\ar[d]^-{\id\varepsilon\id^2}\ar@/_1.5pc/[dd]_-{(*)}\\
  & 
  & 
  FFY'\ar[d]^-{m\id}\\
  Y\ar[r]^-{f} & 
  FY'\ar@{=}[r] & 
  FY'
}
\end{align*}
The composite along the upper horizontal and right-hand vertical maps is $s^0d^1f$ and the composite along the bottom horizontal maps is $f$; the diagram commutes verifies that $s^0d^1=\id$. Similarly, on the level of hom-sets (simplicial degree 0), let's verify that $s^0d^0=
\id$ on $\Hombold(Y,FY')$. Start with $\function{f}{Y}{FY'}$ and consider the commutative diagram
\begin{align*}
\xymatrix{
  Y\ar[r]^-{m}\ar[d]^-{\eta\id} & 
  \K FY\ar[r]^-{\id^2 f}\ar[d]^-{\eta\id^3} & 
  \K FFY'\ar[r]^-{\id m\id}\ar[d]^-{\eta\id^4} & 
  \K FY'\ar[r]^-{\eta\id^3}\ar[d]^-{\eta\id^3} & 
  F\K FY'\ar@{=}[d]\\
  FY\ar[r]^-{\id m}\ar@{=}[dd] & 
  F\K FY\ar[r]^-{\id^3 f}\ar[d]^-{\id\varepsilon\id^2}\ar@/_1.5pc/[dd]_-{(*)} & 
  F\K FFY'\ar[r]^-{\id^2 m\id}\ar[d]^-{\id\varepsilon\id^3} & 
  F\K FY'\ar@{=}[r]\ar[d]^-{\id\varepsilon\id^2} & 
  F\K FY'\ar[d]^-{\id\varepsilon\id^2}\\
  & 
  FFY\ar[r]^-{\id^2 f}\ar[d]^-{m\id} & 
  FFFY'\ar[r]^-{\id m\id}\ar[d]^-{m\id^2} & 
  FFY'\ar@{=}[r]\ar[d]^-{m\id}& 
  FFY'\ar[d]^-{m\id}\\
  FY\ar@{=}[r]  & 
  FY\ar[r]^-{\id f} & 
  FFY'\ar[r]^-{m\id} & 
  FY'\ar@{=}[r] & 
  FY'\\
  Y\ar[u]_-{\eta\id} \ar[rr]^-{f}
  &  
  & 
  FY'\ar[u]_-{\eta\id^2}\ar@{=}[r] & 
  FY'\ar@{=}[u]
}
\end{align*}
The composite along the upper horizontal and right-hand vertical maps is $s^0d^0f$ and the composite along the bottom horizontal maps is $f$; the diagram commutes verifies that $s^0d^0=\id$.
\end{rem}

\begin{defn}
\label{defn:realization_sSet}
The \emph{realization} functor $\function{|-|}{\sSet}{\CGHaus}$ for simplicial sets is defined objectwise by the coend $X\mapsto X \times_{\Delta}\Delta^{(-)}$; here, $\Delta^n$ in $\CGHaus$ denotes the topological standard $n$-simplex for each $n\geq 0$ (see \cite[I.1.1]{Goerss_Jardine}).
\end{defn}

\begin{defn}
Let $X,Y$ be pointed spaces. The mapping space functor $\Map$ is defined objectwise by realization
$
  \Map(X,Y)
  :=|\Hombold(X,Y)|
$
of the indicated simplicial set.
\end{defn}

\begin{defn}
Let $Y,Y'$ be $\Kt$-coalgebra spectra. The \emph{mapping spaces} of derived $\Kt$-coalgebra maps $\Hombold_{\coAlgKt}(Y,Y')$ in $\sSet$ and $\Map_{\coAlgKt}(Y,Y')$ in $\CGHaus$ are defined by the restricted totalizations
\begin{align*}
  \Hombold_{\coAlgKt}(Y,Y')
  &:=\Tot^\res\Hombold\bigl(Y,F\Kt^\bullet Y'\bigr)\\
  \Map_{\coAlgKt}(Y,Y')
  &:=\Tot^\res\Map\bigl(Y,F\Kt^\bullet Y'\bigr)
\end{align*}
of the indicated objects.
\end{defn}

\begin{rem}
\label{rem:intuition_behind_resolution}
Note that there are natural zigzags of weak equivalences
\begin{align*}
  \Hombold_{\coAlgKt}(Y,Y')
  \wequiv\holim_{\Delta}\Hombold\bigl(Y,F\Kt^\bullet Y'\bigr)
\end{align*}
\end{rem}

\begin{defn}
\label{defn:derived_K_coalgebra_map}
Let $Y,Y'$ be $\Kt$-coalgebra spectra. A \emph{derived $\Kt$-coalgebra map} $f$ of the form $Y\rarrow Y'$ is any map in $(\sSet)^{\Delta_\res}$ of the form
\begin{align*}
  \functionlong{f}{\Delta[-]}{\Hombold\bigl(Y,F\Kt^\bullet Y'\bigr)}.
\end{align*}
A \emph{topological derived $\Kt$-coalgebra map} $g$ of the form $Y\rarrow Y'$ is any map in $(\CGHaus)^{\Delta_\res}$ of the form
\begin{align*}
  \functionlong{g}{\Delta^\bullet}{\Map\bigl(Y,F\Kt^\bullet Y'\bigr)}.
\end{align*}
The \emph{underlying map} of a derived $\Kt$-coalgebra map $f$ is the map $\function{f_0}{Y}{FY'}$ that corresponds to the map $\function{f_0}{\Delta[0]}{\Hombold(Y,FY')}$. Every derived $\Kt$-coalgebra map $f$ determines a topological derived $\Kt$-coalgebra map $|f|$ by realization.
\end{defn}

\begin{defn}
The \emph{homotopy category} of $\Kt$-coalgebra spectra (compare, \cite[1.15]{Arone_Ching_classification}), denoted $\Ho(\coAlgKt)$, is the category with objects the $\Kt$-coalgebras and morphism sets $[X,Y]_\Kt$ from $X$ to $Y$ the path components
\begin{align*}
  [X,Y]_\Kt := \pi_0\Map_\coAlgKt(X,Y)
\end{align*}
of the indicated mapping spaces.
\end{defn}

\begin{defn}
\label{defn:weak_equivalence_of_K_coalgebras}
A derived $\Kt$-coalgebra map $f$ of the form $Y\rarrow Y'$ is a \emph{weak equivalence} if the underlying map $\function{f_0}{Y}{FY'}$ is a weak equivalence.
\end{defn}

\begin{prop}
\label{prop:weak_equivalence_if_and_only_if_iso_in_homotopy_category}
Let $Y,Y'$ be $\Kt$-coalgebra spectra. A derived $\Kt$-coalgebra map $f$ of the form $Y\rarrow Y'$ is a weak equivalence if and only if it represents an isomorphism in the homotopy category of $\Kt$-coalgebras.
\end{prop}

\section{Background on the derived fundamental adjunction}
\label{sec:derived_fundamental_adjunction}

Here we review the stabilization version of the derived fundamental adjunction; see also Arone-Ching \cite{Arone_Ching_classification} (which is based on ideas described in Hess \cite{Hess}), Ching-Harper \cite{Ching_Harper_derived_Koszul_duality}, and Cohn \cite{Cohn}. 

The derived unit is the map of pointed spaces of the form $X\rarrow\holim_\Delta \mathfrak{C}(\Susp^\infty X)$ corresponding to the identity map $\function{\id}{\Susp^\infty X}{\Susp^\infty X}$; it is tautologically the $\Loopt^\infty\Susp^\infty$-completion map $X\rarrow X^\wedge_{\Loopt^\infty\Susp^\infty}$ studied in Carlsson in \cite{Carlsson_equivariant} (see also Arone-Kankaanrinta \cite{Arone_Kankaanrinta}).

\begin{defn}
\label{defn:derived_counit_map}
The \emph{derived counit map} associated to the fundamental adjunction \eqref{eq:fundamental_adjunction_comparing_spaces_with_coAlgK} is the derived $\Kt$-coalgebra map of the form $\Susp^\infty\holim_\Delta \mathfrak{C}(Y)\rightarrow Y$ with underlying map
\begin{align}
\label{eq:derived_counit_map_of_the_form}
  \Susp^\infty\Tot^\res\mathfrak{C}(Y)\longrightarrow FY
\end{align}
corresponding to the identity map
\begin{align}
\label{eq:derived_identity_map}
  \function{\id}{\Tot^\res\mathfrak{C}(Y)}{\Tot^\res\mathfrak{C}(Y)}
\end{align}
in $\Space_*$, via the adjunctions \cite[5.4]{Blomquist_Harper_integral_chains} and $(\Susp^\infty,\Loop^\infty)$. In more detail, the derived counit map is the derived $\Kt$-coalgebra map defined by the composite
\begin{align}
\label{eq:derived_counit_map}
  \Delta[-]\xrightarrow{(*)}
  &\Hombold\bigl(\Tot^\res\mathfrak{C}(Y),\mathfrak{C}(Y)\bigr)\\
  \notag
  \Iso
  &\Hombold\bigl(\Susp^\infty\Tot^\res\mathfrak{C}(Y),F\Kt^\bullet Y\bigr)
\end{align}
in $(\sSet)^{\Delta_\res}$, where $(*)$ corresponds to the map \eqref{eq:derived_identity_map} in $\Space_*$.
\end{defn}

\begin{prop}
\label{prop:induced_map_on_mapping_spaces_built_from_stabilization}
Let $X,X'$ be pointed spaces.
There are natural morphisms of mapping spaces of the form
\begin{align*}
  \Susp^\infty\colon\thinspace\Hombold(X,X')
  \rarrow&\Hombold_\coAlgKt(\Susp^\infty X,\Susp^\infty X'),\\
  \Susp^\infty\colon\thinspace\Map(X,X')
  \rarrow&\Map_\coAlgKt(\Susp^\infty X,\Susp^\infty X'),
\end{align*}
in $\sSet$ and $\CGHaus$, respectively.
\end{prop}

\begin{prop}
There is an induced functor
\begin{align*}
  \function{\Susp^\infty}{\Ho(\Space_*)}{\Ho(\coAlgKt)}
\end{align*}
which on objects is the map $X\mapsto \Susp^\infty X$ and on morphisms is the map
\begin{align*}
  [X,X']\rarrow [\Susp^\infty X,\Susp^\infty X']_\Kt
\end{align*}
which sends $[f]$ to $[\Susp^\infty f]$.
\end{prop}

\begin{prop}
\label{prop:cosimplicial_resolutions_of_K_coalgebras_respect_adjunction_isos}
Let $X\in\Space_*$ and $Y\in\coAlgKt$. The adjunction isomorphisms associated to the $(\Susp^\infty,\Loop^\infty)$ adjunction induce well-defined isomorphisms
\begin{align*}
  \Hombold\bigl(\Susp^\infty X,F\Kt^\bullet Y\bigr)&\xrightarrow{\Iso}
  \Hombold\bigl(X,\Loopt^\infty\Kt^\bullet Y\bigr)
\end{align*}
of cosimplicial objects in $\sSet$, natural in $X,Y$.
\end{prop}

\begin{prop}
\label{prop:zigzag_of_weak_equivalences_tot_and_stabilization_completion}
If $X$ is a pointed space, then there is a zigzag of weak equivalences
\begin{align*}
  X^\wedge_{\Loopt^\infty\Susp^\infty}\wequiv
  \holim\nolimits_\Delta \mathfrak{C}(\Susp^\infty X)\wequiv
  \Tot^\res\mathfrak{C}(\Susp^\infty X)
\end{align*}
in $\Space_*$, natural with respect to all such $X$.
\end{prop}

\begin{defn}
A pointed space $X$ is \emph{$\Loopt^\infty\Susp^\infty$-complete} if the natural coaugmentation $X\wequiv X^\wedge_{\Loopt^\infty\Susp^\infty}$ is a weak equivalence.
\end{defn}

\begin{prop}
\label{prop:fundamental_adjunction_derived_version}
There are natural zigzags of weak equivalences
\begin{align*}
  \Map_\coAlgKt(\Susp^\infty X,Y)\wequiv
  \Map(X,\holim\nolimits_\Delta \mathfrak{C}(Y))
\end{align*}
in $\CGHaus$; applying $\pi_0$ gives the natural isomorphism $[\Susp^\infty X,Y]_\Kt\Iso[X,\holim_\Delta \mathfrak{C}(Y)]$.
\end{prop}

\begin{proof}
There are natural zigzags of weak equivalences of the form
\begin{align*}
  \Hombold(X,\holim\nolimits_\Delta \mathfrak{C}(Y))
  &\wequiv\Hombold\bigl(X,\Tot^\res \mathfrak{C}(Y)\bigr)\\
  &\Iso\Tot^\res\Hombold\bigl(X,\Loopt^\infty \Kt^\bullet Y\bigr)\\
  &\Iso\Tot^\res\Hombold\bigl(\Susp^\infty X,F\Kt^\bullet Y\bigr)\\
  &\Iso\Hombold_\coAlgKt(\Susp^\infty X,Y)
\end{align*}
in $\sSet$; applying realization, together with \cite[6.14]{Ching_Harper_derived_Koszul_duality} finishes the proof.
\end{proof}

\begin{prop}
\label{prop:formal_adjunction_and_iso_argument}
Let $X,X'$ be pointed spaces. If $X'$ is $\Loopt^\infty\Susp^\infty$-complete and fibrant, then there is a natural zigzag
\begin{align*}
  \Susp^\infty\colon\thinspace\Map(X,X')\xrightarrow{\wequiv}
  \Map_\coAlgKt(\Susp^\infty X,\Susp^\infty X')
\end{align*}
of weak equivalences; applying $\pi_0$ gives the map $[f]\mapsto[\Susp^\infty f]$.
\end{prop}

\bibliographystyle{plain}
\bibliography{SuspensionSpectra}

\end{document}